\documentclass[10pt]{amsart}
\usepackage[margin=1.5in]{geometry}
\usepackage{color}
\newtheorem{theorem}{Theorem}[section]
\newtheorem{lemma}[theorem]{Lemma}
\newtheorem{proposition}{Proposition}[section]

\theoremstyle{definition}
\newtheorem{definition}[theorem]{Definition}

\theoremstyle{remark}
\newtheorem{remark}[theorem]{Remark}

\numberwithin{equation}{section}

\usepackage{graphicx}
\DeclareGraphicsExtensions{.pdf,.png,.jpg}
\begin{document}
	\title[Staionary solution for relativistic BGK model]{Stationary solutions to the boundary value problem for relativistic BGK model in a slab}
	\author[B.-H. Hwang]{Byung-Hoon Hwang}
	\address{Department of Mathematics, Sungkyunkwan University, Suwon 440-746, Republic of Korea}
	\email{bhh0116@skku.edu}

	\author[S.-B. Yun]{Seok-Bae Yun}
	\address{Department of Mathematics, Sungkyunkwan University, Suwon 440-746, Republic of Korea}

	\email{sbyun01@skku.edu}

	
\begin{abstract}
 In this paper, we are concerned with the boundary value problem in a slab for the stationary relativistic BGK model of Marle type, which  is a relaxation model of the relativistic Boltzmann equation.
In the case of fixed inflow boundary conditions, we  establish the existence of unique stationary solutions.
 \end{abstract}
	\maketitle

\section{Introduction}
In this paper, we address the existence of stationary solutions for a relativistic BGK model defined on a unit interval:
\begin{equation}\label{SRBGK}
	q_1\frac{\partial f}{\partial x}=w\big(J_f-f\big),\quad (x,q)\in [0,1]\times \mathbb{R}^3,
\end{equation}
endowed with a fixed inflow data at the boundary:
$$
f(0,q)=f_L(q)~~ \text{for}~~ q_1>0,\qquad f(1,q)=f_R(q)~~ \text{for}~~ q_1<0,
$$
for some given functions $f_L$ and $f_R$.
The momentum distribution function $f(x,q)$ represents the number density of relativistic particles at position $x\in [0,1]$ with momentum $q\in \mathbb{R}^3$. On the r.h.s of (\ref{SRBGK}), $w$ is a collision frequency, and $J_f$ denotes the local relativistic Maxwellian defined by
$$
J_f=\frac{n}{M(\beta)}e^{-\beta\left( \sqrt{1+|u|^2}\sqrt{1+|q|^2}-u\cdot q\right)},
$$
where $M(\beta)$ is
$$
M(\beta)=\int_{\mathbb{R}^3}e^{-\beta\sqrt{1+|p|^2}}dp,
$$
and the proper particle density $n$, velocity four-vector $(\sqrt{1+|u|^2},u)$ and the equilibrium temperature $1/\beta$ are defined by the following relations:
(in the following, $q_0$ denotes $\sqrt{1+|q|^2}$ for  $q\in\mathbb{R}^3$).
\begin{align}\label{macro fields}
\begin{split}
 	&\hspace{0.2cm}n^2=\biggl(\int_{\mathbb{R}^3}f dq\biggl)^2-\sum_{i=1}^3\biggl(\int_{\mathbb{R}^3}f\frac{q_i}{q_0} dq\biggl)^2,\cr
 	& n\sqrt{1+|u|^2}=\int_{\mathbb{R}^3}fdq,\quad n u=\int_{\mathbb{R}^3}f\frac{q}{q_0}dq,\cr
 	&\hspace{1.3cm}\frac{K_1}{K_2}(\beta)=\frac{1}{n}\int_{\mathbb{R}^3}f \frac{1}{q_0}dq,
\end{split}
\end{align}
where $K_i$ denotes a modified Bessel function of the second kind:
\[
K_i(\beta)=\int_0^\infty \cosh(ir)e^{-\beta\cosh(r)}dr \quad(i=1,2).
\]
It is shown in  \cite{BCNS} that $K_1/K_2$ is strictly increasing and, therefore,
the last identity in (\ref{macro fields}) uniquely determines $\beta$.

The kinetic theory of relativistic particles began with J\"{u}ttner \cite{Jut1} in 1911 when he derived a relativistic version of the Maxwellian distribution, which is often called the J\"{u}ttner equilibrium. The relativistic generalization of the celebrated Boltzmann equation was made by Lichnerowicz and Marrot in 1941 \cite{LM}.

The complicated structure of the relativistic collision operator, however, has long been a major obstacle in the application of the relativistic Boltzmann equation to various flow problems. To circumvent this difficulty, two types of relaxation in time approximation were suggested \cite{AW,Mar2, Mar3} to develop a numerically amenable model equation which still shares essential features of the collision operator such as the conservation laws and H-theorem. The first one was proposed by Marle \cite{Mar2, Mar3} where the macroscopic fields are represented using the Eckart decomposition, and the other one by Anderson and Witting \cite{AW} where the Landau-Lifshitz decomposition was employed for the representation of the macrosocipic fields.


The relativistic BGK models then have been widely used for various purposes \cite{AP,AW,CK,CK2,CKT,FRS,HM,HMPS,Kremer,Kremer2,Majorana,Mar1,Mar3,MBHS,Stru2,TH,TI}, but rigorous mathematical studies have just started and lots of issues still remain to be addressed. In 2012, Bellouquid et al considered the determination of equilibrium parameters, various formal scaling limits  and the analysis of the linearized problem of the Marle model in \cite{BCNS}. Then, the existence of mild solutions and asymptotic stability of the Marle model near global relativistic equilibrium were proved  in \cite{BNU}. To the best knowledge of authors, these two works are the only mathematical literatures treating the relativistic BGK model analytically.

Much more have been done for the relativistic Boltzmann equation. We refer to \cite{B} for local existence, and \cite{D,DE,GS1,GS2,Strain1,Strain Zhu} for global existence near equilibrium. Momentum regularity was established in  \cite{Guo Strain Momentum} leading to the global existence for the relativistic Vlasov-Maxwell-Boltzmann equation near equilibriums. For the existence of renormalized solutions with general data, see \cite{Dud3,Jiang1,Jiang2}. Existence and some moment estimates for the spatially homogeneous relativistic Boltzmann equation can be found in \cite{LR,Strain Yun}. We also refer to \cite{Cal,Strain2} for Newtonian limits and \cite{SS} for hydrodynamic limits.



\subsection{Main result}
In this paper, we consider the stationary relativistic BGK model of Marle type posed on a bounded interval with fixed inflow boundary data at both ends, and establish the existence of unique stationary solutions.
\begin{definition}\label{mild}
A non-negative function $f\in L^1\left([0,1]\times\mathbb{R}^3\right)$ is called a mild solution of \eqref{SRBGK} if
\begin{align*}
f(x,q)&=\left(e^{-\frac{w}{|q_{1}|}x}f_L+\frac{w}{|q_1|}\int_0^xe^{-\frac{w}{|q_1|}(x-y)}J_fdy\right)1_{q_1>0}\cr
&+\left(e^{-\frac{w}{|q_{1}|}(1-x)}f_R+\frac{w}{|q_1|}\int_x^1e^{-\frac{w}{|q_1|}(y-x)}J_fdy\right)1_{q_1<0}.
\end{align*}
\end{definition}
\noindent For brevity, we denote
\begin{align*}
f_{LR}&=f_L1_{q_1>0}+f_R1_{q_1<0},\cr
f^e_{LR}&=e^{-\frac{w}{|q_{1}|}x}f_{L}1_{q_1>0}+e^{-\frac{w}{|q_{1}|}(1-x)}f_{R}1_{q_1<0}.
\end{align*}
We then define quantities $a_l, a_u$ and $\lambda$ by
\begin{align*}
a_l=\int_{\mathbb{R}^3} e^{-\frac{w}{|q_1|}}f_{LR}\frac{1}{q_0^2} dq,\qquad a_u=2\int_{\mathbb{R}^3} f_{LR}dq,
\end{align*}
and
\begin{align}\label{lambda}
	\lambda =\biggl(\int_{\mathbb{R}^3} f^e_{LR} \frac{1}{q_0}dq\biggl)\biggl(\int_{\mathbb{R}^3} f^e_{LR} dq\int_{\mathbb{R}^3} f^e_{LR} \frac{1}{q^2_0}dq\biggl)^{-\frac{1}{2}}.
\end{align}
Our main result is as follows:
\begin{theorem}\label{main}
Suppose the inflow boundary data $f_{LR}$ is non-negative and belongs to $L^1(\mathbb{R}^3_q)$.
Assume further that $a_l>0$.
Then, we can find $\varepsilon>0$ such that if $w<\varepsilon$, then there exists a unique mild solution $f$ to \eqref{SRBGK} such that
$$
\int_{\mathbb{R}^3} f\frac{1}{q^2_0}dq\geq a_l,\qquad \int_{\mathbb{R}^3} fdq\le a_u, \qquad \frac{K_1}{K_2}(\beta)\le \sqrt{\lambda}.
$$
\end{theorem}
\begin{remark}
(1) $f_{LR}\in L^1(\mathbb{R}^3_q)$ guarantees that $a_l, a_u$ are well-defined, and the condition $a_l>0$ guarantees that the divisor in the ratio $\lambda$ is non-zero.\newline
\noindent (2) Considering the conditions under which the identity holds in the H\"{o}lder inequality, we see that $a_{\ell}>0$ implies $0<\lambda<1$. Therefore, $\sqrt{\lambda}$ is
strictly less than 1 (see the following paragraph below to see why this condition is important).
\end{remark}
To prove our main result, we adapt and make a relativistic extension of the argument in \cite{BY} where one of the authors considered  the stationary problem of the classical ellipsoidal BGK model for classical particles.
The relativistic nature of the equation complicates the problem at virtually every point, and makes the adaptation nontrivial.
One of the key differences arises in  the way the relativistic counterpart of the local temperature $1/\beta$ is defined, which is implicitly defined through a nonlinear functional relation:
\begin{align*}
 \frac{K_1}{K_2}(\beta)=\frac{1}{n}\int_{\mathbb{R}^3}f\frac{1}{q_0}dq.
\end{align*}
This implies that we need to control $(K_1/K_2)(\beta)$  to get a proper bound on $\beta$. 
In view of this, we first  note that we already have some control on it: $0<(K_1/K_2)(\beta)<1$,
which holds trivially by the definitions of $K_1$ and $K_2$.  This trivial bound, however, gives no information on the size of $\beta$ since
$(K_1/K_2)^{-1}(1)=\infty$. And without the information on the size of $\beta$,
we cannot guarantee that our solution space is invariant under our solution map, which is essential to close the
fixed point argument.
Therefore, we need to bound $K_1/K_2$ by a constant that is  strictly less than 1 (see the remark 1.3 (2)). This is accomplished in Lemma \ref{222} using the
following estimate controlling the relativistic Maxwellian by the collision frequency and the boundary data:
 \begin{align*}
&\int_{q_1>0}\frac{w}{|q_1|}\int_0^xe^{-\frac{w}{|q_1|}(x-y)}J_fdy\frac{dq}{q_0}+	\int_{q_1<0}\frac{w}{|q_1|}\int_x^1e^{-\frac{w}{|q_1|}(y-x)}J_fdy\frac{dq}{q_0}\cr
&\hspace{1.4cm}\le \frac{16C_1}{C^2_2}\left(2w\ln\frac{1}{w}+(1+e)w+\frac{\sqrt{2}w^2}{C_2}e^{-\frac{C_2}{\sqrt{2}w}}\right),
\end{align*}
which is established in Lemma \ref{phi2}.\newline


The paper is organized as follows: In Section 2, we define the solution space and present several preliminary technical estimates.
In Section 3, we define the solution operator and  show that it maps the solution space into itself.
In Section 4, we show that the solution operator is a contraction mapping on the solution space under our assumptions.
\newline
%
%
%
%

\section{Estimates in solution space}
We define our solution space $\Omega$ by
$$
\Omega=\left\{f(x,q)\in L^1\left([0,1]\times\mathbb{R}^3\right)~|~ f~ \text{satisfies}~ (\mathcal{A}) \right\}
$$
where the property $(\mathcal{A})$ denotes
$$
\qquad f\ge 0,\qquad a_l\le \int_{\mathbb{R}^3} f\frac{1}{q^2_0}dq,\qquad \int_{\mathbb{R}^3} fdq\le a_u, \qquad \frac{K_1}{K_2}(\beta)\le \sqrt{\lambda}.
$$
We will show that the solution to the boundary value problem (\ref{SRBGK}) is given as a unique fixed point in $\Omega$ of a solution operator, which will be defined later. First, we need to establish several
preliminary estimates.

\begin{lemma}\label{11}
Let $f\in\Omega$, then the macroscopic quantities $n, u, \beta$ constructed from $f$ by the relation (\ref{macro fields}) satisfy
$$
a_l\le n\le a_u,\qquad |u|\le \frac{\sqrt{2}a_u}{a_l},\qquad  \frac{a_l}{a_u}\leq \frac{K_1}{K_2}(\beta)\leq  \sqrt{\lambda}.
$$
\end{lemma}
\begin{proof}First, we see from the definition of $n$ that
$$
n^2=\biggl(\int_{\mathbb{R}^3}f dq\biggl)^2-\sum_{i=1}^3\biggl(\int_{\mathbb{R}^3}f\frac{q_i}{q_0} dq\biggl)^2\le \biggl(\int_{\mathbb{R}^3}f dq\biggl)^2\le a_u^2.
$$
For the lower bound of $n$, we employ the H\"{o}lder inequality as follows:
\begin{align*}
\sum_{i=1}^3\biggl(\int_{\mathbb{R}^3}f\frac{q_i}{q_0} dq\biggl)^2&\leq\sum_{i=1}^3\int_{\mathbb{R}^3}f dq\int_{\mathbb{R}^3}f \frac{q_i^2}{q_0^2}dq\cr
&=\int_{\mathbb{R}^3}f dq\sum_{i=1}^3\int_{\mathbb{R}^3}f \frac{q_i^2}{q_0^2}dq\cr
&=\int_{\mathbb{R}^3}f dq\int_{\mathbb{R}^3}f \frac{|q|^2}{q_0^2}dq.
\end{align*}
to get
\begin{align}\label{holder}
\begin{split}
n^2&=\biggl(\int_{\mathbb{R}^3}f dq\biggl)^2-\sum_{i=1}^3\biggl(\int_{\mathbb{R}^3}f\frac{q_i}{q_0} dq\biggl)^2\cr
&\geq\biggl(\int_{\mathbb{R}^3}f dq\biggl)^2-\int_{\mathbb{R}^3}f dq\int_{\mathbb{R}^3}f \frac{|q|^2}{q_0^2}dq\cr
&=\biggl(\int_{\mathbb{R}^3}f dq\biggl)\biggl(\int_{\mathbb{R}^3}f dq-\int_{\mathbb{R}^3}f \frac{|q|^2}{q_0^2}dq\biggl)\cr
&=\int_{\mathbb{R}^3}f dq\int_{\mathbb{R}^3}f \frac{1}{q^2_0}dq\cr
&\ge \biggl(\int_{\mathbb{R}^3}f \frac{1}{q^2_0}dq\biggl)^2\cr
&= a_l^2.
\end{split}
\end{align}
Using this, we can bound $|u|$ from above:
\[
\displaystyle|u|=\frac{|nu|}{n}\le \frac{1}{a_l}\bigg|\int_{\mathbb{R}^3} f\frac{q}{q_0} dq\bigg|\le \frac{\sqrt{2}}{a_l}\int_{\mathbb{R}^3} f\frac{|q|}{\sqrt{1+|q|^2}} dq\le \frac{\sqrt{2}}{a_l}\int_{\mathbb{R}^3} fdq \le \frac{\sqrt{2}a_u}{a_l},
\]	
and bound $K_1/{K_2}$ from below:
$$
\sqrt{\lambda}\ge\frac{K_1}{K_2}(\beta)=\frac{1}{n}\int_{\mathbb{R}^3}f \frac{1}{q_0}dq\ge \frac{1}{a_u}\int_{\mathbb{R}^3}f \frac{1}{q^2_0}dq\ge \frac{a_l}{a_u}.
$$
\end{proof}	
\begin{lemma}\label{22}
For $f\in \Omega$, there exist constants $C_{1}$ and $C_2$ depending on $a_l,~ a_u$ such that
$$
J_f \le C_{1}e^{-C_{2}\sqrt{1+|q|^2}}.
$$
\end{lemma}
\begin{proof}
It is shown in \cite{BCNS} that $(K_1/K_2)(\beta)$ is strictly increasing, and the range is $[0,1)$ for $\beta\in(0,\infty)$. Therefore, from the result
of Lemma \ref{11}:
$$
\frac{a_l}{a_u}\le\frac{K_1}{K_2}(\beta)\le \sqrt{\lambda},
$$
we can specify the range of $\beta$ as follows:	
\begin{equation}\label{b} \beta_l\equiv\left(\frac{K_1}{K_2}\right)^{-1}\left(\frac{a_l}{a_u}\right)\le\beta
\le\left(\frac{K_1}{K_2}\right)^{-1}\left(\sqrt{\lambda}\right)\equiv\beta_u.
\end{equation}
This, together with  the upper bound of $n$ in Lemma \ref{11} and the fact that $M(\beta)$ is a decreasing function,  gives
\begin{equation}\label{nm}
\frac{n}{M(\beta)}\le \frac{a_u}{M(\beta_u)}.
\end{equation}	
On the other hand, we recall $0\leq |u|\leq \sqrt{2}a_u/a_l$ from Lemma \ref{11}, and use the fact that $h(x)\equiv\sqrt{1+x^2}-x$ is a non-negative decreasing function to
conclude that
\begin{align}\begin{split}\label{uq}
\sqrt{1+|u|^2}\sqrt{1+|q|^2}-u\cdot q&\ge  \sqrt{1+|u|^2}\sqrt{1+|q|^2}-|u||q|\cr
&\ge (\sqrt{1+|u|^2}-|u|)\sqrt{1+|q|^2}\cr
&\ge C_0\sqrt{1+|q|^2},
\end{split}
\end{align}
for
$$
C_0=h\bigg(\frac{\sqrt{2}a_u}{a_l}\bigg)=\sqrt{
	1+\bigg(\frac{\sqrt{2}a_u}{a_l}\bigg)^2}-\frac{\sqrt{2}a_u}{a_l}>0.
$$
Combining \eqref{nm} and \eqref{uq}, we  obtain the desired result:
$$
J_f=\frac{n}{M(\beta)}e^{-\beta\left(\sqrt{1+|u|^2}\sqrt{1+|q|^2}-u\cdot q\right)}\le~ \frac{a_u}{M(\beta_u)}e^{-\beta_lC_0\sqrt{1+|q|^2}}\equiv C_1e^{-C_2\sqrt{1+|q|^2}}.
$$
\end{proof}
\begin{lemma}\label{phi2}
Let $f\in\Omega$. Assume $0<w<1$. Then we have
\begin{align*}
\int_{q_1>0}\frac{w}{|q_1|}\int_0^xe^{-\frac{w}{|q_1|}(x-y)}J_fdydq+	&\int_{q_1<0}\frac{w}{|q_1|}\int_x^1e^{-\frac{w}{|q_1|}(y-x)}J_fdydq\cr
&\le \frac{16C_1}{C^2_2}\left(2w\ln\frac{1}{w}+(1+e)w+\frac{\sqrt{2}w^2}{C_2}e^{-\frac{C_2}{\sqrt{2}w}}\right).
\end{align*}
\end{lemma}
\begin{proof} We only consider $\int_{q_1>0}$ to avoid the repetition. From Lemma \ref{22},
\begin{align}\label{r.h.s}
\begin{split}
\int_{q_1>0}\frac{w}{q_1}\int_0^xe^{-\frac{w}{q_1}(x-y)}J_fdydq&\le C_1\int_{q_1>0}\frac{w}{q_1}\int_0^xe^{-\frac{w}{q_1}(x-y)}e^{-C_2\sqrt{1+|q|^2}}dydq\cr&\le \frac{8C_1}{C^2_2}\int_{q_1>0}\frac{w}{q_1}\int_0^xe^{-\frac{w}{q_1}(x-y)}e^{-\frac{C_2}{\sqrt{2}}|q_1|}dydq_1.
\end{split}
\end{align}
Here we used
\begin{align*}
\int e^{-C_2\sqrt{1+|q|^2}}dq_2dq_3&\le \int e^{-\frac{C_2}{\sqrt{2}}(|q_1|+|q_2|+|q_3|)}dq_2dq_3 \cr
&=e^{-\frac{C_2}{\sqrt{2}}|q_1|}\int e^{-\frac{C_2}{\sqrt{2}}(|q_2|+|q_3|)}dq_2dq_3\cr
&=\frac{8}{C_2^2}e^{-\frac{C_2}{\sqrt{2}}|q_1|}.
\end{align*}
Now we split integral on the r.h.s of (\ref{r.h.s}) into the following two parts:
$$
\int_{q_1>0}=\underbrace{\int_{0<q_1\le\frac{1}{w}}}_{I}+\underbrace{\int_{q_1>\frac{1}{w}}}_{II}.
$$
$\bullet$ (Estimate for $I$): We split $I$ further as
\begin{align*}
I&=\int_{0<q_1\le\frac{1}{w}}e^{-\frac{C_2}{\sqrt{2}}|q_1|}(1-e^{-\frac{w}{q_1}x})dq_1 \cr
&=\biggl\{\underbrace{\int_{0<q_1\le w}}_{I_{1}}+\underbrace{\int_{w<q_1\le \frac{1}{w}}}_{I_{2}}\biggl\}e^{-\frac{C_2}{\sqrt{2}}|q_1|}(1-e^{-\frac{w}{q_1}x})dq_1.
\end{align*}
For $I_{1}$ we have
$$
I_1\le	\int_{0<q_1\le w}1-e^{-\frac{w}{q_1}x}dq_1\le 	\int_{0<q_1\le w}dq_1\le w.
$$
For $I_2$, we use Taylor expansion to estimate
\begin{align*}
I_2&=\int_{w<q_1\le\frac{1}{w}}e^{-\frac{C_2}{\sqrt{2}}|q_1|}\left(\frac{w}{q_1}x-\frac{1}{2!}\left(\frac{w}{q_1}x\right)^2+\frac{1}{3!}\left(\frac{w}{q_1}x\right)^3-\frac{1}{4!}\left(\frac{w}{q_1}x\right)^4+\cdots\right)dq_1\cr
&\le \int_{w<q_1\le\frac{1}{w}}\left(\frac{w}{q_1}+\frac{1}{2!}\left(\frac{w}{q_1}\right)^2+\frac{1}{3!}\left(\frac{w}{q_1}\right)^3+\frac{1}{4!}\left(\frac{w}{q_1}\right)^4+\cdots\right)dq_1\cr
&=2w\ln \frac{1}{w}+\frac{w}{2!}(1-w^2)+\frac{w}{2\cdot3!}(1-w^4)+\frac{w}{3\cdot4!}(1-w^6)+\cdots+\cdots\cr
&\le 2w\ln \frac{1}{w}+w\bigg(1+\frac{1}{2!}+\frac{1}{3!}+\frac{1}{4!}+\cdots\bigg) \cr
&= 2w\ln\frac{1}{w}+we.
\end{align*}
Therefore, we have
\begin{align}\label{I}
I\le 2w\ln\frac{1}{w}+(1+e)w.
\end{align}
$\bullet$ (Estimate for $II$): Since $0\leq x\leq1$, we can bound $II$ as
\begin{align}\label{II}
\begin{split}
II
\le w^2 \int_{q_1>\frac{1}{w}}e^{-\frac{C_2}{\sqrt{2}}|q_1|}dq_1
= \frac{\sqrt{2}w^2}{C_2}e^{-\frac{C_2}{\sqrt{2}w}}.
\end{split}
\end{align}
Combining \eqref{I} and \eqref{II}, we get the desired result:
\begin{equation*}
\int_{q_1>0}\frac{w}{q_1}\int_0^xe^{-\frac{w}{q_1}(x-y)}J_fdydq\le \frac{8C_1}{C^2_2}\left(2w\ln\frac{1}{w}+(1+e)w+\frac{\sqrt{2}w^2}{C_2}e^{-\frac{C_2}{\sqrt{2}w}}\right).
\end{equation*}
\end{proof}
\section{$\Phi$ maps  $\Omega$ into itself}
For $f\in\Omega$, we define our solution operator $\Phi(f)$ as follows:
\begin{align*}
\Phi(f)(x,q)&\equiv\left(e^{-\frac{w}{|q_{1}|}x}f_L+\frac{w}{|q_1|}\int_0^xe^{-\frac{w}{|q_1|}(x-y)}J_fdy\right)1_{q_1>0}\cr
&+\left(e^{-\frac{w}{|q_{1}|}(1-x)}f_R+\frac{w}{|q_1|}\int_x^1e^{-\frac{w}{|q_1|}(y-x)}J_fdy\right)1_{q_1<0}\cr
&\equiv\Phi^+(f)1_{q_1>0}+\Phi^-(f)1_{q_1<0}.
\end{align*}
The main goal of this section is to prove that if $f$ belongs to $\Omega$, $\Phi(f)$ also belongs to $\Omega$:
\begin{proposition}\label{prop1}
The solution operator maps the solution space $\Omega$ into itself.
\end{proposition}
The proof of this proposition is given in the following two lemmas.
\begin{lemma}\label{111} Let $f\in \Omega$. Then we have
$$
\Phi(f)\ge 0.
$$
\end{lemma}
\begin{proof}
From Lemma \ref{11} and \eqref{b}, we see that
$$
\frac{n_f}{M(\beta)}\ge \frac{a_l}{M(\beta_{l})}>0
$$
which gives the positivity of $J_f$. We thus have
\begin{align}\label{phi}
\begin{split}
\Phi(f)&=\left(e^{-\frac{w}{|q_{1}|}x}f_L+\frac{w}{|q_1|}\int_0^xe^{-\frac{w}{|q_1|}(x-y)}J_fdy\right)1_{q_1>0}\cr
&+\left(e^{-\frac{w}{|q_{1}|}(1-x)}f_R+\frac{w}{|q_1|}\int_x^1e^{-\frac{w}{|q_1|}(y-x)}J_fdy\right)1_{q_1<0}\cr
&\ge e^{-\frac{w}{|q_{1}|}x}f_L1_{q_1>0}+e^{-\frac{w}{|q_{1}|}(1-x)}f_R1_{q_1<0}\cr
&\ge 0.
\end{split}
\end{align}
\end{proof}

\begin{lemma}\label{222} Let $f\in\Omega$. Then, for sufficiently small $w$, $\Phi(f)$ satisfies
$$
\int_{\mathbb{R}^3} \Phi(f)\frac{1}{q^2_0}dq\ge a _l,\qquad \int_{\mathbb{R}^3}\Phi(f)dq\le a_u,\qquad 	\frac{1}{n_{\Phi}}\int_{\mathbb{R}^3} \Phi(f)\frac{1}{q_0}dq\le \sqrt{\lambda}
$$
where $n_\Phi$ denotes the proper particle density with respect to $\Phi$:
$$
n_\Phi^2=\biggl(\int_{\mathbb{R}^3}\Phi(f) dq\biggl)^2-\sum_{i=1}^3\biggl(\int_{\mathbb{R}^3}\Phi(f)\frac{q_i}{q_0} dq\biggl)^2.
$$
\end{lemma}
\begin{proof}
(1) We have from \eqref{phi}
\begin{align*}
\Phi(f)\frac{1}{q^2_0}&\ge e^{-\frac{w}{|q_{1}|}x}f_L\frac{1}{q^2_0}~1_{q_1>0}+e^{-\frac{w}{|q_{1}|}(1-x)}f_R\frac{1}{q^2_0}~1_{q_1<0}\cr
&\ge e^{-\frac{w}{|q_{1}|}}f_L\frac{1}{q^2_0}~1_{q_1>0}+e^{-\frac{w}{|q_{1}|}}f_R\frac{1}{q^2_0}~1_{q_1<0}\cr
&= e^{-\frac{w}{|q_{1}|}}f_{LR}\frac{1}{q^2_0},
\end{align*}
yielding
$$
\int_{\mathbb{R}^3} \Phi(f)\frac{1}{q^2_0}dq\ge \int_{\mathbb{R}^3} e^{-\frac{w}{|q_{1}|}}f_{LR}\frac{1}{q^2_0}dq =a_l.
$$	
(2) We recall Lemma \ref{phi2} to see
\begin{align*}
\int_{\mathbb{R}^3}\Phi(f)dq&=\int_{\mathbb{R}^3} e^{-\frac{w}{|q_{1}|}x}f_L1_{q_1>0}+e^{-\frac{w}{|q_{1}|}(1-x)}f_R1_{q_1<0}dq\cr
&+\int _{\mathbb{R}^3}\frac{w}{|q_1|}\int_0^xe^{-\frac{w}{|q_1|}(x-y)}J_fdy1_{q_1>0}+\frac{w}{|q_1|}\int_x^1e^{-\frac{w}{|q_1|}(y-x)}J_fdy1_{q_1<0}dq\cr
&\le\int_{\mathbb{R}^3} f_{LR}dq+\frac{16C_1}{C^2_2}\left(2w\ln\frac{1}{w}+(1+e)w+\frac{\sqrt{2}w^2}{C_2}e^{-\frac{C_2}{\sqrt{2}w}}\right).
\end{align*}
We then choose a sufficiently small $w$  so that
$$
\frac{16C_1}{C^2_2}\left(2w\ln\frac{1}{w}+(1+e)w+\frac{\sqrt{2}w^2}{C_2}e^{-\frac{C_2}{\sqrt{2}w}}\right)\le \int_{\mathbb{R}^3} f_{LR} dq
$$
to get
$$
\int_{\mathbb{R}^3} \Phi(f)dq\le 2\int_{\mathbb{R}^3} f_{LR} dq=a_u.
$$
(3) Estimating similarly as in \eqref{holder}, we get
\begin{align*}
n_{\Phi}&\geq\biggl(\int_{\mathbb{R}^3}\Phi(f) dq\int_{\mathbb{R}^3}\Phi(f) \frac{1}{q^2_0}dq\biggl)^{\frac{1}{2}}.
\end{align*}
But we have from the positivity of $J_f$
\begin{align*}
f^e_{LR}(x,q)&=e^{-\frac{w}{|q_{1}|}x}f_L1_{q_1>0}+e^{-\frac{w}{|q_{1}|}(1-x)}f_R1_{q_1<0}\cr
&\le \left(e^{-\frac{w}{|q_{1}|}x}f_L+\frac{w}{|q_1|}\int_0^xe^{-\frac{w}{|q_1|}(x-y)}J_fdy\right)1_{q_1>0}\cr
&+\left(e^{-\frac{w}{|q_{1}|}(1-x)}f_R+\frac{w}{|q_1|}\int_x^1e^{-\frac{w}{|q_1|}(y-x)}J_fdy\right)1_{q_1<0}\cr
&=\Phi(f)(x,q),
\end{align*}
so that we can bound $n_{\Phi}$ further from below as
\begin{align*}
n_{\Phi}\geq\biggl(\int_{\mathbb{R}^3} f^e_{LR}dq\int_{\mathbb{R}^3} f^e_{LR}\frac{1}{q^2_0}dq\biggl)^{\frac{1}{2}}.
\end{align*}
Therefore, we get
\begin{align*}
\begin{split}
\frac{1}{n_{\Phi}}\int_{\mathbb{R}^3} \Phi(f)\frac{1}{q_0}dq&\leq\biggl(\int_{\mathbb{R}^3}\Phi(f) dq\int_{\mathbb{R}^3}\Phi(f) \frac{1}{q^2_0}dq\biggl)^{-\frac{1}{2}}\int_{\mathbb{R}^3} \Phi(f)\frac{1}{q_0}dq\cr
&\leq\biggl(\int_{\mathbb{R}^3} f^e_{LR}dq\int_{\mathbb{R}^3} f^e_{LR}\frac{1}{q^2_0}dq\biggl)^{-\frac{1}{2}}\int_{\mathbb{R}^3} \Phi(f)\frac{1}{q_0}dq,
\end{split}
\end{align*}
which, in view of the definition of $\lambda$ in (\ref{lambda}), leads to
\begin{align}\label{phi6}
\begin{split}
\frac{1}{n_{\Phi}}\int_{\mathbb{R}^3} \Phi(f)\frac{1}{q_0}dq
&\leq \lambda\biggl(\int_{\mathbb{R}^3} f^e_{LR}\frac{1}{q_0}dq\biggl)^{-1}\int_{\mathbb{R}^3} \Phi(f)\frac{1}{q_0}dq.
\end{split}
\end{align}
Now, since
\begin{align*}
\int_{\mathbb{R}^3}\Phi(f)\frac{1}{q_0}dq&=\int_{\mathbb{R}^3} f^e_{LR} \frac{1}{q_0}dq+\int_{q_1>0}\frac{w}{|q_1|}\int_0^xe^{-\frac{w}{|q_1|}(x-y)}J_fdy\frac{1}{q_0}dq\cr
&+\int_{q_1<0}\frac{w}{|q_1|}\int_x^1e^{-\frac{w}{|q_1|}(y-x)}J_fdy\frac{1}{q_0}dq,
\end{align*}
Lemma \ref{phi2} implies
\begin{align*}
\int_{\mathbb{R}^3}\Phi(f)\frac{1}{q_0}dq\leq\int_{\mathbb{R}^3} f^e_{LR} \frac{1}{q_0}dq+\frac{16C_1}{C^2_2}\left(2w\ln\frac{1}{w}+(1+e)w+\frac{\sqrt{2}w^2}{C_2}e^{-\frac{C_2}{\sqrt{2}w}}\right).
\end{align*}
We then note that, as $w$ decreases,
$$
f^e_{LR}=e^{-\frac{w}{|q_{1}|}x}f_L1_{q_1>0}+e^{-\frac{w}{|q_{1}|}(1-x)}f_R1_{q_1<0}
$$
increases, which enables one to find $w$ sufficiently small such that
$$
\frac{16C_1}{C^2_2}\left(2w\ln\frac{1}{w}+(1+e)w+\frac{\sqrt{2}w^2}{C_2}e^{-\frac{C_2}{\sqrt{2}w}}\right)\le \left(\frac{1}{\sqrt{\lambda}}-1 \right)\int_{\mathbb{R}^3} f^e_{LR} \frac{1}{q_0}dq,
$$
yielding
\begin{align*}
\begin{split}
\int_{\mathbb{R}^3}\Phi(f)\frac{1}{q_0}dq&\le\int_{\mathbb{R}^3} f^e_{LR} \frac{1}{q_0}dq+\frac{16C_1}{C^2_2}\left(2w\ln\frac{1}{w}+(1+e)w+\frac{\sqrt{2}w^2}{C_2}e^{-\frac{C_2}{\sqrt{2}w}}\right)\cr
&\le \frac{1}{\sqrt{\lambda}}\int_{\mathbb{R}^3} f^e_{LR} \frac{1}{q_0}dq.
\end{split}
\end{align*}
Inserting this into \eqref{phi6} gives the desired result.
\end{proof}

\section{Contraction mapping}
In this section, we establish the Lipschitz continuity of our solution operator. We first need to set up preliminary
computations. The following lemma can be found  in \cite{BCNS}, but we provide a detailed proof for reader's convenience.
\begin{lemma}\label{bessel} For the modified Bessel function of the second kind $K_i(\beta)$, the following holds
$$
\biggl(\frac{K_{1}}{K_{2}}\biggl)^{\prime}(\beta) =\frac{3}{\beta}\frac{K_{1}}{K_{2}}(\beta)+\biggl(\frac{K_{1}}{K_{2}} \biggl)^{2}(\beta)-1.
$$
	
\end{lemma}
\begin{proof}
We recall that
$$
K_i(\beta)=\int_0^\infty \cosh(ir)e^{-\beta\cosh(r)}dr,
$$
 and use change of variable $x=\sinh r$ to get
\begin{align*}
K_{0}(\beta)&=\int_{0}^{\infty}\exp\biggl\{-\beta
\cosh(r)\biggl\}dr=\int_{0}^{\infty}\frac{1}{\sqrt{1+x^{2}}}\exp\biggl\{-\beta
\sqrt{1+x^{2}}\biggl\}dx,\cr
K_{1}(\beta)&=\int_{0}^{\infty}\cosh(r)\exp\biggl\{-\beta \cosh(r)\biggl\}dr=\int_{0}^{\infty}\exp\biggl\{-\beta \sqrt{1+x^{2}}\biggl\}dx, \cr
K_{2}(\beta)&=\int_{0}^{\infty}\cosh(2r)\exp\biggl\{-\beta
\cosh(r)\biggl\}dr=\int_{0}^{\infty}\frac{2x^{2}+1}{\sqrt{1+x^{2}}}\exp\biggl\{-\beta
\sqrt{1+x^{2}}\biggl\}dx.
\end{align*}
We then observe that
\begin{align*}
(K_{1})^{\prime}(\beta)&=-\int_{0}^{\infty}\sqrt{1+x^{2}}\exp\biggl\{-\beta \sqrt{1+x^{2}}\biggl\}dx\cr
&=-\frac{1}{2}\biggl(K_{2}(\beta)+K_{0}(\beta)\biggl)\cr
&=-\biggl(\frac{1}{\beta}K_{1}(\beta)+K_{0}(\beta)\biggl)
\end{align*}
and
\begin{align*}
(K_{2})^{\prime}(\beta)&=-\int_{0}^{\infty}(2x^{2}+1)\exp\biggl\{-\beta \sqrt{1+x^{2}}\biggl\}dx\cr
&=\frac{2}{\beta}\int_{0}^{\infty}x\sqrt{1+x^{2}}\frac{d}{dx}\biggl(\exp\biggl\{-\beta \sqrt{1+x^{2}}\biggl\}\biggl)dx-K_{1}(\beta)\cr
&=-\frac{2}{\beta}\int_{0}^{\infty}\frac{2x^{2}+1}{\sqrt{1+x^{2}}}\exp\biggl\{-\beta \sqrt{1+x^{2}}\biggl\}dx-K_{1}(\beta)\cr
&=-\frac{2}{\beta}K_{2}(\beta)-K_{1}(\beta)\cr
&=\biggl(-\frac{4}{\beta^{2}}-1\biggl)K_{1}(\beta)-\frac{2}{\beta}K_{0}(\beta).
\end{align*}
Here we used
\begin{align*}
K_{2}(\beta)&=\int_{0}^{\infty}\frac{2x^{2}+1}{\sqrt{1+x^{2}}}\exp\biggl\{-\beta \sqrt{1+x^{2}}\biggl\}dx\cr
&=\int_{0}^{\infty}\frac{2x^{2}}{\sqrt{1+x^{2}}}\exp\biggl\{-\beta \sqrt{1+x^{2}}\biggl\}dx+K_{0}(\beta)\cr
&=\int_{0}^{\infty}-\frac{2x}{\beta}\frac{d}{dx}\biggl(\exp\biggl\{-\beta \sqrt{1+x^{2}}\biggl\}\biggl)dx+K_{0}(\beta)\cr
&=\frac{2}{\beta}\int_{0}^{\infty}\exp\biggl\{-\beta \sqrt{1+x^{2}}\biggl\}dx+K_{0}(\beta)\cr
&=\frac{2}{\beta}K_{1}(\beta)+K_{0}(\beta).
\end{align*}
These identities then give
\begin{equation*}
(K_{1})^{\prime}(\beta)K_{2}(\beta)-K_{1}(\beta)(K_{2})^{\prime}(\beta)=\frac{3}{\beta}K_{1}(\beta)K_{2}(\beta)+(K_{1})^{2}(\beta)-(K_{2})^{2}(\beta),
\end{equation*}	
which, upon dividing both sides by  $(K_{2})^2(\beta)$, gives the desired result.
\end{proof}
The following lemma shows that the r.h.s of the identity in Lemma \ref{bessel} is strictly positive. The proof can be
found in \cite{BCNS}.
\begin{lemma}\cite{BCNS}\label{bessel2} For the modified Bessel function of the second kind, the following inequality holds
\begin{align*}
		\frac{3}{\beta}\frac{K_{1}}{K_{2}}(\beta)+\biggl(\frac{K_{1}}{K_{2}}\biggl)^{2}(\beta)-1
		\geq\ell(\beta),
	\end{align*}
	where $\ell(\beta)$ is defined by
	\begin{align*}
		\ell(\beta)\equiv\begin{cases}\frac{2-\beta}{(\beta+2)^{2}} & 0<\beta<2\cr \frac{3(6656\beta^{4}+2419\beta^{3}+726)}{(128\beta^{3}+240\beta^{2}+105\beta-66)^{2}}&\beta\geq2.\end{cases}
	\end{align*}
Note that $\ell(\beta)$ is strictly positive.
\end{lemma}
\begin{lemma}\label{JJ} Let $f,g\in\Omega$, then, for sufficiently small $w$, there exist positive constants $C_8$ and $C_9$ such that
$$
|J_f-J_g|\le C_9e^{-C_8\sqrt{1+|q|^2}}\|f-g\|_{L^1_q},
$$
where $\|\cdot\|_{L^1_q}$ denotes the usual $L^1$ norm:
\[
\|f\|_{L^1_q}=\int_{\mathbb{R}^{3}}|f(q)|dq.
\]
\end{lemma}
\begin{proof}
For the convenience of computation, we introduce a new variable $\alpha$ (see \cite{BCNS}) defined by
$$
\alpha =\frac{1}{n}\int_{\mathbb{R}^3}f\frac{1}{q_0}dq.
$$
Due to the monotonicity of $K_1/K_2$ established in Lemma \ref{bessel} and Lemma \ref{bessel2}, there is a one-to-one correspondence between $\alpha$ and $\beta$:
\[
\beta=\mathcal{X}(\alpha)=\biggl( \frac{K_1}{K_2}\biggl)^{-1}(\alpha).
\]
In view of this, we consider  $J(n,u,\beta)$ as a functional of $(n,u,\alpha)$, and apply the mean value theorem to get
\begin{align}\label{J}
J(n_f,u_f,\alpha_f)-J(n_g,u_g,\alpha_g)=\nabla_{n,u,\alpha} J(\theta)\cdot\big(n_f-n_g,u_f-u_g,\alpha_f-\alpha_g\big),
\end{align}	
for some $0\leq \theta\leq 1$, where the abbreviate notation $J(\theta)$ denotes
\[
J(\theta)=J\big((1-\theta)n_f+\theta n_g,(1-\theta)u_{f}+\theta u_g,(1-\theta)\alpha_{f}+\theta \alpha_g\big).
\]
We need to estimate $\nabla_{n,u,\alpha} J$ and $(n_f-n_g,u_f-u_g,\alpha_f-\alpha_g\big)$. \newline

\noindent(1) Estimates for $\nabla_{n,u,\alpha} J$: A direct computation gives
\begin{align}\label{return}
\begin{split}
&\frac{\partial J}{\partial n}=\frac{1}{n}J,\qquad  		 \nabla_u J=\beta\biggl(q-\frac{\sqrt{1+|q|^2}}{\sqrt{1+|u|^2}}u\biggl)J,\cr
&\frac{\partial J}{\partial \alpha}=-\frac{\partial\beta}{\partial\alpha}\biggl(\frac{M^{\prime}(\beta)}{M(\beta)}+\sqrt{1+|q|^2}\sqrt{1+|u|^2}-u\cdot q\biggl)J.
\end{split}
\end{align}
Using Lemma \ref{11}, Lemma \ref{22} and \eqref{b}, we can show that $\frac{\partial J}{\partial n}$ and $\nabla_u J$ are bounded as
\begin{align*}
\left|\frac{\partial J}{\partial n}\right|&\le \frac{C_{1}}{a_l}e^{-C_{2}\sqrt{1+|q|^2}},\cr
\left|\nabla_u J\right|&\le\beta_u\biggl(|q|+\sqrt{1+|q|^2}\frac{|u|}{\sqrt{1+|u|^2}}\biggl)|J| \cr
&\le 2C_{1}\beta_u\sqrt{1+|q|^2}e^{-C_{2}\sqrt{1+|q|^2}}\cr
&\le 2C_{1}\beta_uC_3e^{-C_4\sqrt{1+|q|^2}}.
\end{align*}
The estimate for $	\frac{\partial }{\partial \alpha}J$ is more involved. First, we use differentiation rule for inverse functions and Lemma \ref{bessel} to get
\begin{align*}
\frac{\partial \beta}{\partial \alpha}&=\frac{\partial}{\partial\alpha}	\biggl(\frac{K_{1}}{K_{2}}\biggl)^{-1}(\alpha)
=\frac{1}{\big(\frac{K_{1}}{K_{2}}\big)^{\prime}(\beta)}
=\frac{1}{\frac{3}{\beta}\frac{K_{1}}{K_{2}}(\beta)+\big(\frac{K_{1}}{K_{2}}\big)^{2}(\beta)-1}.
\end{align*}
We then recall Lemma \ref{bessel2} that
\begin{align*}
\frac{3}{\beta}\frac{K_{1}}{K_{2}}(\beta)+\biggl(\frac{K_{1}}{K_{2}}\biggl)^{2}(\beta)-1
\geq\ell(\beta).
\end{align*}
Note that $\ell(\beta)$ is a strictly positive function. Therefore, we can conclude that the continuous function $\big(\frac{K_{1}}{K_{2}}\big)^{\prime}(\beta)$
(since the modified Bessel function of the second kind is continuous)  is strictly positive
on a closed and bounded interval $[\beta_l,\beta_u]$. This implies that $\big(\frac{K_{1}}{K_{2}}\big)^{\prime}(\beta)$ posses a strictly positive minimum on $[\beta_l,\beta_u]$.
If we denote it by $1/C_5>0$, we have
\begin{align*}
\frac{3}{\beta}\frac{K_{1}}{K_{2}}(\beta)+\biggl(\frac{K_{1}}{K_{2}}\biggl)^{2}(\beta)-1\geq \frac{1}{C_5}.
\end{align*}
In conclusion, we obtain
\begin{align*}
\Big|\frac{\partial\beta}{\partial\alpha}\Big|&=\Big|\frac{3}{\beta}\frac{K_{1}}{K_{2}}(\beta)+\biggl(\frac{K_{1}}{K_{2}}\biggl)^{2}(\beta)-1\Big|^{-1}\cr
&=\Big(\frac{3}{\beta}\frac{K_{1}}{K_{2}}(\beta)+\biggl(\frac{K_{1}}{K_{2}}\biggl)^{2}(\beta)-1\Big)^{-1}\cr
&\leq C_5.
\end{align*}
On the other hand, it is clear that there exists a constant $C_6>0$ such that
$$
\frac{M^{\prime}(\beta)}{M(\beta)}=-\frac{\int_{\mathbb{R}^3} \sqrt{1+|q|^2}\exp\{-\beta\sqrt{1+|q|^2}\}dq}{\int_{\mathbb{R}^3} \exp\{-\beta\sqrt{1+|q|^2}\}dq}<C_6,
$$
when $\beta$ lies in a closed and bounded range: $\beta\in[\beta_l,\beta_u]$. We return back to (\ref{return}) with these estimates to get
\begin{align*}
\left| \frac{\partial J}{\partial \alpha}\right|&\le C_5\left(C_6+2\sqrt{1+|q|^2}\sqrt{1+4a^2_u/a^2_l}\right)C_1e^{-C_2\sqrt{1+|q|^2}}\cr
&\le C_5C_6C_1e^{-C_2\sqrt{1+|q|^2}}+2C_5C_3\sqrt{1+4a^2_u/a^2_l}C_1e^{-C_4\sqrt{1+|q|^2}}\cr
&\le C_7e^{-C_8\sqrt{1+|q|^2}}
\end{align*}
(2) Estimates on $(n_f-n_g ,u_f-u_g ,\alpha_f-\alpha_g )$ \newline
$\bullet~n_f-n_g$: From Lemma \ref{11}, we get
\begin{align*}
|n_{f}-n_{g}|&=\left|\frac{n^2_f-n^2_g}{n_{f}+n_{g}}\right|\cr
&\le\frac{1}{n_{f}+n_{g}}\biggl\{\left(\int_{\mathbb{R}^3}|f-g| dq\right)\left(\int_{\mathbb{R}^3}|f+g| dq\right)\cr
&~-\sum_{i=1}^3\left(\int_{\mathbb{R}^3}|f-g|\frac{q_i}{q_0} dq\right)\left(\int_{\mathbb{R}^3}|f+g|\frac{q_i}{q_0} dq\right)\biggl\}\cr
&\le \frac{1}{2a_l}\biggl(2a_u\int_{\mathbb{R}^3}|f-g| dq+\sum_{i=1}^32a_u\int_{\mathbb{R}^3}|f-g|dq\biggl)\cr
&\le \frac{4a_u}{a_l}\|f-g\|_{L^1_q}.
\end{align*}
$\bullet~u_f-u_g$: Using the above estimate, Lemma \ref{11} and
\begin{align*}
|n_fu_f-n_gu_g|= \left|\int_{\mathbb{R}^3}(f-g)\frac{q}{q_0}dq\right|\le\sqrt{2} \int_{\mathbb{R}^3}|f-g|dq= \sqrt{2}\|f-g\|_{L^1_q},
\end{align*}
we compute
\begin{align*}
|u_{f}-u_{g}|&=\left|\frac{n_g(n_fu_f)-n_f(n_gu_g)}{n_{f}n_{g}}\right|\cr
&=\left|\frac{n_g(n_fu_f-n_gu_g)-n_gu_g(n_f-n_g)}{n_{f}n_{g}}\right|\cr
&\le\frac{\sqrt{2}}{a^2_l}\big(a_u \|f-g\|_{L^1_q}+\sqrt{2}a_u\frac{4a_u}{a_l}\|f-g\|_{L^1_q}\big) \cr
&\le\left( \frac{\sqrt{2}a_u}{a_l^2}+\frac{8a^2_u}{a^3_l}\right)\|f-g\|_{L^1_q}.
\end{align*}
$\bullet~\alpha_f-\alpha_g$:  We can estimate similarly as in the previous cases:
\begin{align*}		
|\alpha_{f}-\alpha_{g}|&=\frac{1}{n_fn_g}\biggl|n_g\int_{\mathbb{R}^3}f\frac{1}{q_0}dq-n_f\int_{\mathbb{R}^3}g\frac{1}{q_0}dq\biggl|\cr
&=\frac{1}{n_fn_g}\biggl|n_g\int_{\mathbb{R}^3}(f-g)\frac{1}{q_0}dq-(n_f-n_g)\int_{\mathbb{R}^3}g\frac{1}{q_0}dq\biggl|\cr
&\le \frac{1}{a^2_l}\big(a_u\|f-g\|_{L^1_q}+a_u\frac{4a_u}{a_l}\|f-g\|_{L^1_q}\big)\cr
&\le \left(\frac{a_u}{a_l^2}+\frac{4a^2_u}{a^3_l}\right)\|f-g\|_{L^1_q}.
\end{align*}
Combining all the estimates we obtained so far, we get
\begin{align*}
\left|\frac{\partial J}{\partial n}(n_f-n_g)\right|&\le \frac{4a_uC_{1}}{a^2_l}e^{-C_{8}\sqrt{1+|q|^2}}\|f-g\|_{L^1_q},\cr
\left|\nabla_u J\cdot (u_{f}-u_{g}) \right|&\le 2C_{1}\beta_u C_3\left( \frac{\sqrt{2}a_u}{a_l^2}+\frac{8a^2_u}{a^3_l}\right)e^{-C_{8}\sqrt{1+|q|^2}}\|f-g\|_{L^1_q},\cr
\left|\frac{\partial J}{\partial \alpha} (\alpha_{f}-\alpha_{g})\right|&\le C_7\left(\frac{a_u}{a_l^2}+\frac{4a^2_u}{a^3_l}\right)e^{-C_8\sqrt{1+|q|^2}}\|f-g\|_{L^1_q}.
\end{align*}
Thus \eqref{J} can be estimates as
\begin{align*}
|J(n_f,u_f,\alpha_f)-J(n_g,u_g,\alpha_g)|&\le  C_9e^{-C_8\sqrt{1+|q|^2}}\|f-g\|_{L^1_q},
\end{align*}
where the constant $C_9$ is given by
\begin{align*}
	C_9&= \frac{a_u}{a^2_l}\left(4C_1+  2C_{1}\beta_u C_3\left( \sqrt{2}+\frac{8a_u}{a_l}\right)+C_7\left(1+\frac{4a_u}{a_l}\right)\right)
\end{align*}
This gives the desired result.
\end{proof}
The following proposition, together with Proposition \ref{prop1} completes the proof of Theorem \ref{main}.
\begin{proposition} $\Phi(f)$ is a contraction mapping on $\Omega$ for sufficiently small $w$. That is, we can take
$w$ sufficiently small such that there exists $0<\alpha<1$ satisfying
\[
\sup_x\|\Phi(f)-\Phi(g)\|_{L^1_q}\leq\alpha \sup_x\|f-g\|_{L^1_q}.
\]
\end{proposition}
\begin{proof}
We have for $f,g\in \Omega$
\begin{align*}
&\int_{\mathbb{R}^3} |\Phi(f)-\Phi(g)|dq\cr
&\qquad=\int_{q_1>0}|\Phi(f)-\Phi(g)|dq+\int_{q_1<0}|\Phi(f)-\Phi(g)|dq \cr
&\qquad\leq\int_{q_1>0}\frac{w}{q_1}\int_0^xe^{-\frac{w}{q_1}(x-y)}|J_f-J_g|dydq+\int_{q_1<0} \frac{w}{|q_1|}\int_x^1e^{-\frac{w}{|q_1|}(y-x)}|J_f-J_g|dydq.
\end{align*}
Then, using Lemma \ref{phi2} and Lemma \ref{JJ}, we can control the last term as follows:
\begin{align*}
& \int_{q_1>0}\frac{w}{q_1}\int_0^xe^{-\frac{w}{q_1}(x-y)}|J_f-J_g|dydq+\int_{q_1<0} \frac{w}{|q_1|}\int_x^1e^{-\frac{w}{|q_1|}(y-x)}|J_f-J_g|dydq\cr
&\hspace{1.2cm}\le C_9\biggl\{\int_{q_1>0}\frac{w}{q_1}e^{-C_8\sqrt{1+|q|^2}}\int_0^xe^{-\frac{w}{q_1}(x-y)}dydq\cr
&\hspace{1.2cm}+\int_{q_1<0} \frac{w}{|q_1|}e^{-C_8\sqrt{1+|q|^2}}\int_x^1e^{-\frac{w}{|q_1|}(y-x)}dydq\biggl\}\|f-g\|_{L^1_q}\cr
&\hspace{1.2cm}\le \frac{16C_9}{C^2_8}\left(2w\ln\frac{1}{w}+(1+e)w+\frac{\sqrt{2}w^2}{C_8}e^{-\frac{C_8}{\sqrt{2}w}}\right)\|f-g\|_{L^1_q}.
\end{align*}
For sufficiently small $w$, this gives the desired results.
\end{proof}
\noindent{\bf Acknowledgement}
This research was supported by Basic Science Research Program through the National Research Foundation of Korea(NRF) funded by the Ministry of Education(NRF-2016R1D1A1B03935955)

\end{document}